\documentclass[reqno]{amsart}

\usepackage{amsmath,amsthm,amssymb}
\usepackage{graphicx}
\usepackage{multicol}
\usepackage{enumerate}
\usepackage[all]{xy}
\usepackage{color}

\newcommand{\jump}[1]{\ensuremath{[\![#1]\!]} }

\newtheorem{Theorem}{Theorem}[section]
\newtheorem{Proposition}[Theorem]{Proposition}
\newtheorem{Lemma}[Theorem]{Lemma}
\newtheorem{Corollary}[Theorem]{Corollary}
\newtheorem{Remark}[Theorem]{Remark}
\newtheorem{Example}[Theorem]{Example}
\newtheorem{Definition}[Theorem]{Definition}

\def\hugesymbol#1{\mbox{\strut\rlap{\smash{\Huge$#1$}}\quad}}

\title{A function determined by a hypersurface of positive characteristic}
\author[Kosuke Ohta]{Kosuke Ohta (Meiji University)}



\begin{document}

\maketitle

\begin{abstract}
Let 
$R=k\jump{X_1, \dots ,X_{n+1}}$
be a formal power series ring over a perfect field $k$ of prime characteristic $p>0$, and let 
$\mathfrak{m} = (X_1 , \dots , X_{n+1})$ be the maximal ideal of $R$.
Suppose 
$0\neq  f \in\mathfrak{m}$. 
In this paper,
we introduce a function $\xi_{f}(x)$ associated with a hypersurface defined on the closed interval
$[0,1]$ in $\mathbb{R}$.
The Hilbert-Kunz function and the F-signature of a hypersurface appear as
the values of our function $\xi_{f}(x)$ on the interval's endpoints.
The F-signature of the pair, denoted by $s(R,f^{t})$, was defined in \cite{BST13}.
Our function $\xi_{f}(x)$ is integrable, and 
the integral
$\displaystyle\int_{t}^{1}\xi_{f}(x)dx$
is just $s(R,f^{t})$ for any $t\in[0,1]$. 
\end{abstract}

\section{Introduction}\label{Introduction}

In this paper,
we introduce a function $\xi_{f}(x)$ associated with a hypersurface defined on the closed interval
$[0,1]$ in $\mathbb{R}$.
The Hilbert-Kunz function and the F-signature of a hypersurface appear as
the values of our function $\xi_{f}(x)$ on the interval's endpoints.

\begin{Definition}
\begin{rm}
Let $(R,\mathfrak{n},k)$ be a $d$-dimensional Noetherian local ring of prime characteristic $p>0$. 
The $Hilbert$-$Kunz$ $multiplicitiy$ of $R$ is 
\[
   e_{HK}(R) = 
   \displaystyle\lim_{e\to\infty}
   \frac{\ell(R/\mathfrak{n}^{[p^e]})}{p^{ed}},
\]
where $\ell(R/\mathfrak{n}^{[p^e]})$ 
is the length of $R/\mathfrak{n}^{[p^e]}$, 
and $\mathfrak{n}^{[p^e]}$ 
is the ideal generated by all the $p^e$-th powers of elements of
$\mathfrak{n}$. 
Monsky showed that this limit always esists 
(\cite{Mo83}). 
\end{rm}
\end{Definition}

Let $(R,\mathfrak{n},k)$ be a $d$-dimensional Noetherian local ring of prime characteristic $p>0$, 
and let $F:R\rightarrow R$ be the Frobenius map,
that is 
$F(x)=x^{p}$ for any $x \in R$.
Taking a positive integer $e>0$, 
we obtain the ring $R$,
denoted by $F^e_*R$,
viewed as an $R$-module via the $e$-th Frobenius map.
The element $s$ in $F^{e}_{*}R$ is sometimes denoted by $F^{e}_{*}(s)$.
For $a,e \in R$, we have $a\cdot F^{e}_{*}(s)=F^{e}_{*}(a^{p^{e}}s)$.
We define the F-signature as follows.

\begin{Definition}
\begin{rm}
Let $(R,\mathfrak{n},k)$ be a $d$-dimensional complete Cohen-Macaulay local ring of prime characteristic $p>0$. 
Assume that $R$ is an F-finite ring, namely, 
the Frobenius map $F:R \rightarrow R$ is finite. 
Suppose
$F^e_*R \simeq R^{\oplus a_e} \oplus M_e$
with some integer $a_{e}$ and some $R$-module $M_{e}$, 
where $M_e$ has no free direct summands. 
The number $a_e$ is called $e$-$th$ $Frobenius$ $splitting$ $number$ of $R$.
Then, 
\[
   s(R) = \displaystyle\lim_{e\to\infty}\frac{a_e}{p^{ed}}
\]
is called the $F$-$signature$ of $R$. 
Tucker showed this limit always exists 
(\cite{T12}, Theorem $4.9$).
\end{rm}
\end{Definition}


In the rest of this paper,
let $n\ge 1$.  
 Let 
 $R=k\jump{X_1, \dots ,X_{n+1}}$ 
 be a formal power series ring over a perfect field $k$ of prime characteristic $p>0$, and let 
 $\mathfrak{m} = (X_1 , \dots , X_{n+1})$ be the maximal ideal of $R$.
 Suppose 
$0\neq  f \in\mathfrak{m}$. 
Rings of the form $R/(f)$ are called ``$n$-$dimensional$ $hypersurface$''.

\begin{Definition}
\begin{rm}
We define
\[
 M_{e,\,t} = \displaystyle\frac{(f^t) + \mathfrak{m}^{[p^e]}}{(f^{t+1}) +  
  \mathfrak{m}^{[p^e]}}
  \simeq \displaystyle\frac{R}{[(f^{t+1})+\mathfrak{m}^{[p^e]}]:f^t} ,
\]
where 
$e \ge 0$ and $t \ge 0$ 
are integers. 
\end{rm}
\end{Definition} 

Then we have, for any $t \ge 0$, the surjection
\[
M_{e,\, t} \rightarrow M_{e,\, t+1}
\]
because 
$[(f^{t+1})+\mathfrak{m}^{[p^e]}]:f^t
\subset
[(f^{t+2})+\mathfrak{m}^{[p^e]}]:f^{t+1}
\subset
R$. 

Let $\overline{R}=R/\mathfrak{m}^{[p^e]}$. 
Then, remark that
$M_{e,\, t}= f^t\overline{R} / f^{t+1}\overline{R}$. 

\begin{Definition}
\begin{rm}
We define 
\[
   C_{e,\,t} = \displaystyle\frac{\ell_R(M_{e,\,t})}{p^{en}} ,
\]
where 
$\ell_R(M_{e,\,t})$ 
is the length as an $R$-module.
\end{rm}
\end{Definition} 
Then we have
\begin{equation}\label{1}
p^e \ge C_{e,\,0} \ge C_{e,\,1} \ge C_{e,\,2} \ge \cdots \ge C_{e,\,p^e-1} \ge C_{e,\,p^e}
=C_{e,\,p^e+1}=\cdots =0.
\end{equation}
A sequence of functions $\{\xi_{f,e}: [0,1]\to {\mathbb R} \}_{e\ge 0}$ is defined by
\[
  \xi_{f,e}(x) = 
  \begin{cases}
     C_{e,\,\lfloor xp^e\rfloor} 
          & (0\le x<1)\\
    C_{e,\,p^e-1} 
          & (x=1)
  \end{cases}, 
\]
where 
$\lfloor xp^e \rfloor = \max{\{ a \in {\mathbb Z} | xp^e \ge a \}}$ 
is the floor function. 
By the definition, we have 
$\displaystyle\int_0^1 \xi_{f,e}(x) dx= 1$ because

\begin{eqnarray*}
     \displaystyle\int_0^1 \xi_{f,e}(x) dx 
     &=& \frac{1}{p^e} 
     \big( C_{e,\,0} + C_{e,\,1} + C_{e,\,2} + \cdots + C_{e,\,p^e-1} \big) \\ 
     &=& \frac{1}{p^e}\times \frac{1}{p^{en}}
     \big( \ell_R(M_{e,\,0}) 
     +  \ell_R(M_{e,\,1}) + \cdots +  \ell_R(M_{e,\,p^e-1}) \big) \\
    &=& \frac{1}{p^{e(n+1)}} \ell_R( R/{\mathfrak m}^{[p^{e}]}) \\
    &=&  \frac{1}{p^{e(n+1)}}\times p^{e(n+1)} \\
    &=& 1.
\end{eqnarray*}

\begin{Definition}
\begin{rm}
We define the function 
$\xi_{f}(x)$ 
by
\[
\xi_{f}(x) = 
  \limsup_{e \to \infty}\xi_{f,e}(x)
\]
for $x \in [0,1]$. 
\end{rm}
\end{Definition}

By the inequalities $\eqref{1}$, $\xi_{f}(x)$ 
is decreasing on $[0,1]$. 
If $\displaystyle\lim_{e \to \infty}\xi_{f,e}(\alpha)$ exists, 
then $\xi_f(\alpha)= \displaystyle\lim_{e \to \infty}\xi_{f,e}(\alpha)$.
The sequence 
$\{C_{e,\,0}\}_e$ 
is increasing by Lemma $\ref{3}$ in section $\ref{main theorem}$. 
\[
\displaystyle\lim_{e \to \infty}C_{e,\,0}
=
\lim_{e \to \infty}\frac{\ell_R(M_{e,\,0})}{p^{en}}
=
\lim_{e \to \infty}\frac{\ell_R(R/(f)+{\mathfrak m}^{[p^e]})}{p^{en}}.
\]
This limit exists and is called the Hilbert-Kunz multiplicity of $R/(f)$, 
denoted by $e_{HK}(R/(f))$. 
Therefore, by $\eqref{1}$,
$\displaystyle\limsup_{e \to \infty}\xi_{f,e}(\alpha)$
is not $+\infty$ for any $\alpha \in [0,1]$.
We shall give an example that
$\displaystyle\lim_{e \to \infty}\xi_{f,e}(\alpha)$
dose not exist for some $f \in R$ and $\alpha \in [0,1]$ in section $\ref{Example}$.
We have
\[
\xi_f(0)
=
e_{HK}(R/(f)).
\]
Therefore, $\xi_{f}(x)$ is a bounded and decreasing function on $[0,1]$. 
In particular, 
$\xi_{f}(x)$ is integrable, 
and has at most countably many points of discontinuity on $[0,1]$.


The main theorem of this paper is the following:

\begin{Theorem}\label{th1}
\begin{rm}
\begin{enumerate}[$1)$]

\item
The function $\xi_{f}(x)$ is decreasing.
There exists a countable subset $C$ of the interval $[0,1]$ 
such that $\xi_{f}(x)$ is continuous at any $\alpha \in [0,1]-C$.
Moreover, $\xi_{f}(x)$ is continuous at $0$ and $1$.

\item
If $\xi_{f}(x)$ is continuous at $\alpha \in [0,1]$, 
then $\displaystyle\lim_{e \to \infty} \xi_{f,e}(\alpha) = \xi_{f}(\alpha)$.

\item
We have $\xi_{f}(0)= e_{HK}(R/(f))$,
and also $\xi_f(1) = s(R/(f))$.

\item 
Suppose that $\xi_{f}(1)=0$, then
${\rm fpt}(f)=\inf\{ \alpha \in [0,1]\;|\; \xi_f(\alpha)=0 \}$
holds, where 
${\rm fpt}(f)=\displaystyle\lim_{e\to \infty}\frac{\mu_f(p^e)}{p^e}$ 
is the F-pure threshold of $f$, 
where $\mu_f(p^e)=\min\{ t\ge 1 \:|\; f^t \in \mathfrak{m}^{[p^e]}\}$.

\item
The function $\xi_{f}(x)$ is integrable,
and we have 
$\displaystyle\int_{\frac{a}{p^e}}^{\frac{a+1}{p^e}} \xi_f(x)dx =
\displaystyle\frac{\ell_{R}(M_{e,\,a})}{p^{e(n+1)}}$
for integers $0 \le a < p^e$.
In particular, 
$\displaystyle\int_{0}^{1} \xi_f(x)dx =1$ holds.

\item
If $R/(f)$ is normal then $\xi'_{f}(0)=0$, where $\xi'_{f}$ is the derivative of $\xi_{f}$.
\end{enumerate}
\end{rm}
\end{Theorem}

\begin{Remark}\label{thRem}
\begin{rm}
By Theorem $1.1$ and Proposition $3.2$ (i) in \cite{BMS09},
we know that above ${\rm fpt}(f)$ is a positive rational number.
Note that F-pure thresholds are defined as the smallest of the F-jumping exponents in \cite{BMS09}.
\end{rm}
\end{Remark}

\begin{Remark}
\begin{rm}
We define the function $\varphi_f(x)$ on $[0,1]$ as follows; 
\[
    \varphi_f(x)= \displaystyle\int_{0}^{x}\xi_{f}(t)dt.
\]
Actually, we have
\[
    \varphi_f(x)= \displaystyle\lim_{e\to \infty}
    \frac{1}{p^{e}}
    \left(  C_{e,\,0} + C_{e,\,1} + \cdots + C_{e,\, \lfloor xp^{e} \rfloor -1} \right).
\]
Since $\xi_f(x)$ is bounded and integrable on $[0,1]$, 
$\varphi_f(x)$ is Lipschitz continuous on $[0,1]$. 
In particular, $\varphi_f(x)$ is continuous on $[0,1]$.
We can rewrite $3)$ and $4)$ in Theorem $\ref{th1}$ as follows;
\begin{enumerate}
\item[3')]
The function $\varphi_f(x)$ is differentiable at $x=0$ and $1$, and 
$\varphi'_f(0) = e_{HK}(R/(f))$ and $\varphi'_f(1) = s(R/(f))$.
\item [4')]
Suppose that $s(R/(f))=0$, then
\[
{\rm fpt}(f)
=
\inf\{ \alpha \in [0,1] \;|\; \varphi_f(\alpha)=1 \}
\]
holds.
\end{enumerate}
Let $(A,{\mathfrak n})$ be an F-finite regular local ring,
and let $g\in A$ be a non-zero element. 
In \cite{BST13}, 
the F-signature of the pair $(A,g^{t})$ for any real number $t \in [0,1]$
is denoted as
\[
s(A,g^{t}) 
= \displaystyle\lim_{e\to\infty}
\frac{1}{p^{e(n+1)}}\ell_{A}\left(\frac{A}{\mathfrak{n}^{[p^e]} : g^{\lceil t(p^{e}-1) \rceil}}\right).
\]
Using $5)$ in Theorem $\ref{th1}$,
we know
\[
1-\varphi_f(x)
=
\displaystyle\int_{t}^{1}\xi_{f}(x)dx
=
s(R,f^{t})
\]
for $t \in [0,1]$.
Moreover,
if we know that $\xi_{f}(x)$ is continuous at $0$ and $1$ (see Theorem $\ref{th1}$ 1)),
we obtain $3)$ in Theorem $\ref{th1}$ immediately from Theorem $4.4$ in \cite{BST13}.
\end{rm}
\end{Remark}

In section $\ref{main theorem}$, we shall prove Theorem \ref{th1}.
The following corollary immediately follows from Theorem \ref{th1} 3) and 5)

\begin{Corollary}\label{1.8}
\begin{rm}
$e_{HK}(R/(f))\times {\rm fpt}(f) \ge 1$.
\end{rm}
\end{Corollary}

\begin{Example}
\begin{rm}
Suppose $R=k\jump{X_1, X_2,\dots , X_{n+1}}$ and $\alpha >0$.
Then $e_{HK}(R/(X^{\alpha}_1))=\alpha$ and ${\rm fpt}(X^{\alpha}_1) = \displaystyle\frac{1}{\alpha}$.
Therefore, if $\tau(f)=X^{\alpha}_1$ for a linear transformation $\tau$ (for example, $f=X_{1}+X_{2}$),
then $e_{HK}(R/(f))\times {\rm fpt}(f) = 1$ and $s(R/(f))=1$
(see section $\ref{Example}$).
We do not know another example that the equality holds in Corollary \ref{1.8}.
\end{rm}
\end{Example}

\begin{Remark}\label{1.10}
\begin{rm}
By Theorem $\ref{th1}$ $1)$, $3)$ and $5)$,
we immediately know that
$e_{HK}(R/(f))=1$ if and only if $s(R/(f))=1$.
These conditions are equivalent to 
that $R/(f)$ is regular by the following results.
\end{rm}
\begin{enumerate}[$1)$]
\item
Let
$S$ be an unmixed local ring of positive characteristic.
Then $e_{HK}(S)=1$ if and only if $S$ 
is regular 
(\cite{WY00}, Theorem $1.5$).

\item
Let $S$ be a reduced F-finite Cohen-Macaulay local ring of positive characteristic. Then $s(S)=1$ if and only if $S$ is regular
(\cite{HL02}, Corollary $16$). 
\end{enumerate}
\end{Remark}

\begin{Remark}
\begin{rm}
Let $m < n = \dim R/(f)$, and set $a_{e}= \ell (M_{e,\, p^{e}-1})$.
Assume that $a_{e} = \alpha p^{em} + o(p^{em})$,
that is $\displaystyle\lim_{e\to\infty}\frac{a_{e}}{p^{em}} = \alpha$.
Let $g_{e} = a_{e} - \alpha p^{em}$.
Then 
\begin{eqnarray*}
\varphi_{f}(1) - \varphi_{f}\left( \frac{p^{e}-1}{p^{e}} \right)
&=&
\sum_{i=0}^{p^{e}-1} \frac{\ell(M_{e,\, i})}{p^{e(n+1)}}
-
\sum_{i=0}^{p^{e}-2} \frac{\ell(M_{e,\, i})}{p^{e(n+1)}} \\
&=&
\frac{\ell(M_{e,\, p^{e}-1})}{p^{e(n+1)}} \\
&=&
\frac{\alpha}{p^{e(n-m+1)}} + \frac{g_{e}}{p^{e(n+1)}}
\end{eqnarray*}
holds.
Let $x=\displaystyle\frac{p^{e}-1}{p^{e}}$.
Since $x-1 = -\displaystyle\frac{1}{p^{e}}$,
we know
\begin{equation}\label{1.12}
\varphi_{f}(x) = \varphi_{f}(1) + (-1)^{n-m}\alpha (x-1)^{n-m+1}+o((x-1)^{n-m+1}).
\end{equation}
Since $\varphi_{f}(x)$ is continuous on $[0,1]$ from Remark $\ref{thRem}$, 
$\varphi_{f}(x)$ has the form of $\eqref{1.12}$ around the point $x=1$. 
Therefore,
if $\varphi_{f}(x)$ is equal to its Taylor series around the point $x=1$,
we obtain that 
\begin{eqnarray*}
\varphi_{f}^{(i)}(1)
&=&
\begin{cases}
  0 & (i=1,2, \dots , n-m)\\
  (-1)^{n-m}(n-m+1)! \alpha & (i=n-m+1)
 \end{cases},\\
\xi_{f}^{(i)}(x)
&=&
\begin{cases}
  0 & (i=1,2, \dots , n-m-1)\\
  (-1)^{n-m}(n-m+1)! \alpha & (i=n-m)
 \end{cases}.
\end{eqnarray*}
\end{rm}
\end{Remark}


\section{Proof of main theorem}\label{main theorem}

Let 
$F:R \rightarrow R$ 
be the Frobenius map 
$a \mapsto a^p$. 
Since $k$ is perfect, 
we have 
$F^{}_{*}R \simeq R^{\oplus p^{n+1}}$,
where $F_{*}R$ stands for $F^{1}_{*}R$. 
Therefore,
$
\displaystyle\frac{(f^{pt})+
{\mathfrak m}^{[p^{e+1}]}}{(f^{pt+p})+
\mathfrak{m}^{[p^{e+1}]}}
= M_{e, \, t} \otimes_{R} F^{}_{*}R
 \simeq  (M_{e,\,t})^{\oplus p^{n+1}}
$ 
for all $e, t \ge 0$.
Consequently, 
\begin{equation}\label{2}
p\times C_{e,\,t} = C_{e+1,\, pt} + C_{e+1,\, pt+1}+\cdots + C_{e+1,\, pt+p-1},
\end{equation}
where the sum on the right-hand side of $\eqref{2}$ has $p$-terms.
That is, 
$C_{e,\,t}$ is the mean of 
$C_{e+1,\, pt}$, $C_{e+1,\, pt+1}$, $\cdots$, $C_{e+1,\, pt+p-1}$. 
Therefore, by $\eqref{1}$ and $\eqref{2}$, 
we obtain the following inequalities immediately.

\begin{Lemma}\label{3}
$C_{e+1,\, pt} \ge C_{e,\,t} \ge C_{e+1,\, pt+p-1}$.
\end{Lemma}

Hence, by $\eqref{1}$ and Lemma $\ref{3}$, we have
\[
\begin{array}{ccccc}
C_{e,\,\lfloor xp^e\rfloor -1 } & \underset{\rm{by}\;\rm{Lemma}\;\ref{3}}{\ge} & 
C_{e+1,\,(\lfloor xp^e\rfloor-1)p+(p-1)} &
\ge & 
C_{e+1,\,\lfloor xp^{e+1}\rfloor -1} \\
\rotatebox[origin=c]{90}{$\le$} & & & & \rotatebox[origin=c]{90}{$\le$} \\\smallskip
 C_{e,\,\lfloor xp^{e}\rfloor} & & & &  C_{e+1,\,\lfloor xp^{e+1}\rfloor}\\
 \rotatebox[origin=c]{90}{$\le$} & & & & \rotatebox[origin=c]{90}{$\le$} \\
C_{e,\,\lceil xp^e \rceil} & \underset{\rm{by}\;\rm{Lemma}\;\ref{3}}{\le} & C_{e+1,\,\lceil xp^e \rceil p} 
&\le& C_{e+1,\,\lceil xp^{e+1} \rceil}
\end{array}
\]
and here, we note that
$\lfloor xp^e \rfloor p \le \lfloor xp^{e+1}\rfloor$ and 
$\lceil xp^e \rceil p \ge \lceil xp^{e+1} \rceil$. 
Therefore, the sequence 
$\{ C_{e,\,\lfloor xp^e\rfloor -1} \}_e$ 
is decreasing,  
the sequence 
$\{ C_{e,\,\lceil xp^e \rceil } \}_e$ 
is increasing, 
and 
$C_{e,\,\lfloor xp^e\rfloor -1} \ge C_{e,\,\lceil xp^e \rceil }$ 
for all $e \ge 0$ by the inequalities $\eqref{1}$. 
Consequently, the limits 
$\displaystyle\lim_{e\to\infty}C_{e,\,\lfloor x p^e\rfloor -1}$
and
$\displaystyle\lim_{e\to\infty}C_{e,\,\lceil x p^e \rceil }$
exist in $\mathbb R$. 
In particular,  
\begin{equation}\label{4}
 C_{e,\,\lfloor \alpha p^e\rfloor -1} 
 \ge
 \displaystyle\lim_{e\to\infty}C_{e,\,\lfloor \alpha p^e\rfloor -1}
 \ge 
 \xi _f(\alpha)
 \ge 
 \displaystyle\lim_{e\to\infty}C_{e,\,\lceil \alpha p^e \rceil }
 \ge
 C_{e,\,\lceil \alpha p^e \rceil }
 \ge 
 0
\end{equation}
holds for any $\alpha \in (0,1]$ and $e$ satisfying $\lfloor \alpha p^e\rfloor -1 \ge 0$.

\begin{Lemma}\label{Cep} 
\begin{rm}
We set $\overline{C}(\alpha)=\displaystyle\lim_{e \to \infty} C_{e,\,\lceil \alpha p^e \rceil}$
for $\alpha\in[0,1]$
and $\underline{C}(\beta)=\displaystyle\lim_{e \to \infty} C_{e,\,\lfloor \beta p^e\rfloor -1}$
for $\beta\in(0,1]$.
\begin{enumerate}[$1)$]
\item
For 
$\alpha \in [0,1]$ 
and any integer 
$i\ge 0$,
$\{ C_{e+1,\,\lceil \alpha p^e \rceil p+i } \}_e$ 
is an increasing sequence. 
The limits
$\displaystyle\lim_{e \to \infty} C_{e+1,\,\lceil \alpha p^e \rceil p+i }$
and
$\displaystyle\lim_{e \to \infty} C_{e,\,\lceil \alpha p^e \rceil + k}$ exist
for any non-negative integers $i$, $k\ge0$.
Furthermore,
\begin{equation}\label{C-over}
\overline{C}(\alpha)
=
\displaystyle\lim_{e \to \infty} C_{e+1,\,\lceil \alpha p^e \rceil p +i}
=
\displaystyle\lim_{e \to \infty} C_{e,\,\lceil \alpha p^e \rceil + k}
\end{equation}
holds.

\item
For 
$\beta \in (0,1]$ and any integer $i> 0$, 
$\{C_{e+1,\,\lfloor \beta p^e\rfloor p-i}\}_e$ 
is a decreasing sequence. 
The limits
$\displaystyle\lim_{e \to \infty} C_{e+1,\,\lfloor \beta p^e\rfloor p-i}$
and
$\displaystyle\lim_{e \to \infty} C_{e,\,\lfloor \beta p^e\rfloor -k}$
exist
for any positive integers $i$, $k>0$.
Furthermore,
\begin{equation}\label{C-under}
\underline{C}(\beta)
=
\displaystyle\lim_{e \to \infty} C_{e+1,\,\lfloor \beta p^e\rfloor p-i}
=
\displaystyle\lim_{e \to \infty} C_{e,\,\lfloor \beta p^e\rfloor -k}
\end{equation}
holds.
\end{enumerate}
\end{rm}
\end{Lemma}

\begin{proof}
Let 
$\alpha \in [0,1]$ and $\beta\in (0,1]$,
and let
$k \ge 0$ and $\ell >0$ be integers.
We know
\[
\begin{cases}
(\lceil \alpha p^e \rceil p+k)p = 
\lceil \alpha p^e \rceil p^2+kp \ge
\lceil \alpha p^{e+1} \rceil p+kp \ge
\lceil \alpha p^{e+1} \rceil p+k  \\

(\lfloor \beta p^e\rfloor p-\ell)p+(p-1) \le 
\lfloor \beta p^e\rfloor p^2-\ell p+(p-1)\ell \le
\lfloor \beta p^{e+1}\rfloor p-\ell
\end{cases}
, 
\]
and therefore
\[
\begin{cases}
C_{e+1,\,\lceil \alpha p^e \rceil p+k} \le
C_{e+2,\,(\lceil \alpha p^e \rceil p+k)p } \le
C_{e+2,\,\lceil \alpha p^{e+1} \rceil p+k } \le
\displaystyle\lim_{e \to \infty}C_{e,\,0}\\

C_{e+1,\, \lfloor \beta p^{e}\rfloor p-\ell} \ge
C_{e+2,\, (\lfloor \beta p^{e}\rfloor p-\ell )p + (p-1)} \ge
C_{e+2,\, \lfloor \beta p^{e+1}\rfloor p-\ell} \ge 0
\end{cases}
\]
by 
$\eqref{1}$ and Lemma $\ref{3}$. 
Hence, 
$\{ C_{e+1,\,\lceil \alpha p^e \rceil p+k } \}_e$ 
is increasing and bounded. 
$\{C_{e+1,\,\lfloor \beta p^e\rfloor p-\ell}\}_e$ 
is decreasing and bounded. 
Therefore, 
$\displaystyle\lim_{e \to \infty}C_{e+1,\,\lceil \alpha p^e \rceil p+k}$ 
and 
$\displaystyle\lim_{e \to \infty}C_{e+1,\,\lfloor \beta p^e\rfloor p-\ell}$ 
exist.

Next, we shall show that
\begin{equation}\label{2.5.1}
\overline{C}(\alpha)
=
\displaystyle\lim_{e \to \infty} C_{e+1,\,\lceil \alpha p^e \rceil p+i}
\end{equation}
\begin{equation}\label{2.5.3}
\underline{C}(\beta)
=
\displaystyle\lim_{e \to \infty} C_{e+1,\,\lfloor \beta p^e\rfloor p-j}
\end{equation}
holds for any integers $0 \le i \le p-1$ and $1 \le j \le p$.
We have 
\[
\begin{cases}
p\times C_{e,\,\lceil\alpha p^e \rceil} = 
C_{e+1,\, \lceil \alpha p^e \rceil p} + 
C_{e+1,\, \lceil \alpha p^e \rceil p+1}+\cdots + 
C_{e+1,\, \lceil \alpha p^e \rceil p+p-1}\\

p\times C_{e,\,\lfloor \beta p^e\rfloor -1} = 
C_{e+1, \,\lfloor \beta p^e\rfloor p -p} + 
C_{e+1, \,\lfloor \beta p^e\rfloor p -(p-1)}+\cdots + 
C_{e+1, \,\lfloor \beta p^e\rfloor p-1}
\end{cases}
\]
by 
$\eqref{2}$. 
Thus, it holds that
\[
\begin{cases}
p\times \displaystyle\lim_{e \to \infty} C_{e,\,\lceil\alpha p^e \rceil} = 
\displaystyle\lim_{e \to \infty}C_{e+1, \,\lceil \alpha p^e \rceil p} + 
\cdots + 
\displaystyle\lim_{e \to \infty}C_{e+1, \,\lceil \alpha p^e \rceil p+p-1}\\

p\times \displaystyle\lim_{e \to \infty}C_{e,\,\lfloor \beta p^e\rfloor -1} = 
\displaystyle\lim_{e \to \infty}C_{e+1, \,\lfloor \beta p^e\rfloor p -p} + 
\cdots + 
\displaystyle\lim_{e \to \infty}C_{e+1,\, \lfloor \beta p^e\rfloor p-1}
\end{cases}.
\]
On the other hand, we have
\[
\begin{cases}
\displaystyle\lim_{e \to \infty} C_{e,\,\lceil\alpha p^e \rceil} = 
\displaystyle\lim_{e \to \infty}C_{e+1,\, \lceil \alpha p^e \rceil p}
\ge
\lim_{e \to \infty}C_{e+1,\, \lceil \alpha p^e \rceil p+1}
\ge
\cdots
\ge
\lim_{e \to \infty}C_{e+1,\, \lceil \alpha p^e \rceil p +p-1}
\\
\displaystyle\lim_{e \to \infty}C_{e,\,\lfloor \beta p^e\rfloor -1} = 
\displaystyle\lim_{e \to \infty}C_{e+1, \,\lfloor \beta p^e\rfloor p-1}
\le
\lim_{e \to \infty}C_{e+1, \,\lfloor \beta p^e\rfloor p-2}
\le
\cdots
\le
\lim_{e \to \infty}C_{e+1, \,\lfloor \beta p^e\rfloor p-p}
\end{cases}
\]
since 
$C_{e,\,\lceil\alpha p^e \rceil} \le 
C_{e+1, \,\lceil \alpha p^e \rceil p} \le
C_{e+1,\,\lceil\alpha p^{e+1} \rceil}$ 
and 
$C_{e,\,\lfloor \beta p^e\rfloor -1} \ge
C_{e+1, \,\lfloor \beta p^e\rfloor p-1} \ge
C_{e+1, \,\lfloor \beta p^{e+1}\rfloor -1}$.
Consequently, we have the equations $\eqref{2.5.1}$ and $\eqref{2.5.3}$.

In order to complete the proof of the assertion $1)$,
we have the inequalities
\begin{eqnarray*}
C_{e,\,\lceil \alpha p^e \rceil + k}
&\le&
C_{e+1,\,(\lceil \alpha p^e \rceil + k)p}\\
&=&
C_{e+1,\,\lceil \alpha p^e \rceil p + kp}\\
&\le&
C_{e+1,\,\lceil \alpha p^{e} \rceil p + k}\\
&\le&
C_{e+1,\,\lceil \alpha p^{e+1} \rceil + k}
\end{eqnarray*}
for any $k\ge 1$.
Hence,
\[
\displaystyle\lim_{e \to \infty}C_{e,\,\lceil \alpha p^e \rceil + k}
=
\displaystyle\lim_{e \to \infty}C_{e+1,\,\lceil \alpha p^{e} \rceil p + k}
\]
holds.
Therefore, we obtain the equation \eqref{C-over}.

In order to complete the proof of the assertion $2)$,
we have the inequalities
\begin{eqnarray*}
C_{e,\,\lfloor \beta p^e \rfloor - k}
&\ge&
C_{e+1,\,(\lfloor \beta p^e \rfloor - k)p +p-1}\\
&=&
C_{e+1,\,\lfloor \beta p^e \rfloor p - (k-1)p -1}\\
&\ge&
C_{e+1,\,\lfloor \beta p^e \rfloor p- k}\\
&\ge&
C_{e+1,\,\lfloor \beta p^{e+1} \rfloor - k}
\end{eqnarray*}
for any $k \ge 2$.
Hence,
\[
\displaystyle\lim_{e \to \infty}C_{e,\,\lfloor \beta p^e \rfloor - k}
=
\displaystyle\lim_{e \to \infty}C_{e+1,\,\lfloor \beta p^e \rfloor p -k}
\]
holds.
Therefore, we obtain the equation \eqref{C-under}.
\end{proof}

\begin{Proposition}\label{lrlim}  
\begin{rm}
\begin{enumerate}[$1)$]
\item
For 
$\alpha \in [0,1)$, 
$\displaystyle{\lim_{x \to \alpha +0} \xi_f(x)=
\lim_{e \to \infty} C_{e,\,\lceil \alpha p^e \rceil }}$ 
holds. 

\item
For 
$\beta \in (0,1]$, 
$\displaystyle{\lim_{x \to \beta -0} \xi_f(x)=
\lim_{e \to \infty} C_{e,\,\lfloor \beta p^e\rfloor -1}}$ 
holds.

\end{enumerate}
$\newline$
In particular, we have 
\[
\begin{cases}
\displaystyle{\lim_{x \to +0} \xi_f(x) =
\lim_{e \to \infty} C_{e,\,0 }} = \xi_f(0) \\

\displaystyle{\lim_{x \to 1-0} \xi_f(x) =
\lim_{e \to \infty} C_{e,\,p^e-1}} = \xi_f(1)
\end{cases},
\]
that is to say that 
$\xi_f(x)$ 
is continuous at 
$x=0$ and $1$.  
\end{rm}
\end{Proposition}

\proof
\begin{rm}
1) First, we show 
$\displaystyle{\lim_{x \to \alpha +0} \xi_f(x) \le
 \lim_{e \to \infty} C_{e,\,\lceil \alpha p^e \rceil }}$. 
Take $x_0 > \alpha$.
For a large enough number 
$e'$, 
we may assume that
$\alpha p^{e'} \le x_0 p^{e'}-2$ 
holds. 
Then, 
$\lceil \alpha p^{e'} \rceil \le \lfloor x_0 p^{e'}\rfloor -1$. 
Hence, by the inequalities 
$\eqref{1}$ and $\eqref{4}$,
\[
  \xi_f(x_0)
  \le
  C_{e',\, \lfloor x_0 p^{e'}\rfloor -1}
  \le
  C_{e',\,\lceil \alpha p^{e'} \rceil }
  \le
  \displaystyle\lim_{e\to \infty}C_{e,\,\lceil \alpha p^{e} \rceil }
\]
as desired.

Next, we shall show the opposite inequality.
By Lemma $\ref{Cep}$ $1)$, we have only to show that 
\[
   \displaystyle{\lim_{x \to \alpha +0} \xi_f(x)
  \ge 
   \lim_{e \to \infty} C_{e,\,\lceil \alpha p^e \rceil +1}}.
\]
For any $e\ge 0$, 
$\alpha < \displaystyle\frac{\lceil \alpha p^{e}\rceil +1}{p^{e}}$. 
Hence, there exists a real number
$x_1 \in {\mathbb R}$ 
such that 
$\alpha < x_1 < \displaystyle\frac{\lceil \alpha p^{e}\rceil +1}{p^{e}}$. 
Then 
$\lceil x_1 p^{e}\rceil \le
\lceil \alpha p^{e}\rceil +1$,
and therefore 
\[
\displaystyle\lim_{x \to \alpha +0}\xi_f(x)
  \ge
  \xi_f(x_1)
  \ge
  C_{e, \,\lceil x_1 p^{e}\rceil}
  \ge
  C_{e,\,\lceil \alpha p^{e}\rceil +1}
\]
for any $e\ge 0$ 
because we have the inequalities $\eqref{1}$ and $\eqref{4}$, 
and $\xi_f(x)$ is decreasing.
Consequenty, 
\[
  \displaystyle{\lim_{x \to \alpha +0} \xi_f(x)
  \ge
  \lim_{e\to\infty}C_{e,\,\lceil \alpha p^{e}\rceil+1}}
\]
as desired. 

2) It is proved in the same way as $1)$.
\end{rm}
\qed

\begin{Remark}\label{2start}
\begin{rm}
From the inequalities $\eqref{1}$, we have 
\[
  C_{e,\,\lfloor \alpha p^e\rfloor -1} 
  \ge 
  \xi_{f,e}(\alpha) 
  =
  C_{e,\,\lfloor \alpha p^e\rfloor} 
  \ge C_{e,\,\lceil \alpha p^e \rceil }
\]
for any $\alpha \in [0,1]$.
Hence, if 
\[
\displaystyle\lim_{e \to \infty}C_{e,\,\lfloor \alpha p^e\rfloor -1} = 
\lim_{e \to \infty}C_{e,\,\lceil \alpha p^e \rceil },
\]
there exists 
$\displaystyle\lim_{e \to \infty}\xi _{f,e}(\alpha)$
in $\mathbb R$, 
and it is eqaul to $\xi_f(\alpha)$. 
\end{rm}
\end{Remark}

\begin{Corollary}\label{2.5}
\begin{rm}
If 
$\xi_f(x)$ 
is continuous at 
$\alpha \in [0,1]$ 
then 
$\displaystyle\lim_{e\to \infty}\xi_{f,e}(\alpha)$ exists, 
so that it is equal to 
$\xi_f(\alpha)$. 
\end{rm}
\end{Corollary}

\proof
The proof is obtained from Remark $\ref{2start}$ immediately.
\qed

We have just shown Theorem $\ref{th1}$ $1)$.

We obtain the following Corollary $\ref{Cor}$ immediately 
from Proposition $\ref{lrlim}$. 

\begin{Corollary}\label{Cor}
\begin{rm}
We define $\varphi_f(x)$ by 
\[
\varphi_f(x)= \displaystyle\int_0^x \xi_f(t) dt
\]
for $x \in [0,1]$. 
Then we have the followings.
\end{rm}
\begin{enumerate}[$1)$]
\item
$\varphi_f(x)$ is differentiable at $0$, and $\varphi'_f(0) = \xi_f(0) = \displaystyle\lim_{e \to \infty} C_{e,\,0 } = e_{HK}(R/(f))$. 
  
\item
  $\varphi_f(x)$ is differentiable at $1$, and $\varphi'_f(1) = \xi_f(1) = \displaystyle\lim_{e \to \infty} C_{e,\,p^e-1 }$. 
\end{enumerate}
\end{Corollary}

Set 
$\mu_f(p^e)=\min\{ t\ge 0  \:|\; f^t \in \mathfrak{m}^{[p^e]}\}$ 
for each $e \ge 0$. 
Since 
$f^{\mu_f(p^e)} \in \mathfrak{m}^{[p^e]}$, 
$f^{\mu_f(p^e) p} \in \mathfrak{m}^{[p^{e+1}]}$. 
Hence 
${\mu_f(p^e) p} \ge {\mu_f(p^{e+1})}$, 
and so 
\[
 1\ge\displaystyle\frac{\mu_f(p^e)}{p^e} \ge \frac{\mu_f(p^{e+1})}{p^{e+1}}\ge0.
\] 
Since $\left\{\displaystyle\frac{\mu_f(p^e)}{p^e}\right\}_{e\ge 0}$ 
is decreasing and bounded below, 
the limit
$\displaystyle\lim_{e \to \infty}\frac{\mu_f(p^e)}{p^e}$ exists in $\mathbb R$, and 
it is called the F-pure threshold of $f$, 
denoted by ${\rm fpt}(f)$. 
It is easy to see that ${\rm fpt}(f) \in (0,1]$,
and ${\rm fpt}(f)=1$ if and only if $\mu_f(p^e)=p^{e}$ for any $e\ge1$.

\begin{Lemma}\label{mu}    
\begin{rm}
$C_{e, \,t} = 0$ 
if and only if 
$t \ge \mu_f(p^e)$.
\end{rm}
\end{Lemma}

\begin{proof}
If $M_{e,\, t}=0$, then 
$M_{e, \,t} = M_{e,\, t+1} = M_{e,\, t+2} = \cdots = M_{e,\, p^e} =0$. 
Hence, 
$f^t \in \mathfrak{m}^{[p^e]}$, 
and so 
$t \ge \mu_f(p^e)$. 
Conversely if 
$t \ge \mu_f(p^e)$, 
then 
$f^{t} \in \mathfrak{m}^{[p^e]}$ holds.
\end{proof}

We start to prove Theorem $\ref{th1}$.
The assertion $1)$ follows from Proposition \ref{lrlim}.
The assertion $2)$ follows from Corollary \ref{2.5}.
The first half of $3)$ follows from the definition of $C_{e,\, 0}$.
Now, we shall show $4)$. 

\begin{proof}
First, we check that 
\[
\inf\{ \alpha\in [0,1] \;|\; \xi_f(\alpha)=0 \} \le {\rm fpt}(f).
\] 
If ${\rm fpt}(f)=1$, then the assertion is easy. 
Assume ${\rm fpt}(f)<1$. 
Let $1>\alpha > {\rm fpt}(f)$. 
Since 
${\rm fpt}(f)=
\displaystyle\inf_{e\ge 0}\left\{\frac{\mu_f(p^e)}{p^e}\right\}$, 
\[
{\rm fpt}(f) \le \displaystyle\frac{\mu_f(p^{e_1})}{p^{e_1}} < \alpha
\]
holds for $e_1\gg 0$.
Then, it holds that 
\begin{eqnarray*}
\xi_f(\alpha) 
&\le& 
\xi_f\left(\displaystyle\frac{\mu_f(p^{e_1})}{p^{e_1}}\right) \\
&=&
\limsup_{e \to \infty} 
  C_{e,\big\lfloor \frac{\mu_f(p^{e_1})}{p^{e_1}}p^e\big\rfloor}\\
&=&
0
\end{eqnarray*}
because, by Lemma $\ref{mu}$,  
\[
C_{e_1 + s,\; \mu_f(p^{e_1})p^s} \le
C_{e_1 + s,\; \mu_f(p^{e_1+s})}=
0
\]
for any integers $s\ge 0$. 
Therefore, 
$\xi_f(\alpha)=0$ for all $\alpha > {\rm fpt}(f)$, as desired. 
Conversely, suppose $\alpha < {\rm fpt}(f)$. 
Hence, 
we have 
$({\rm fpt}(f)-\alpha) p^{e'} \ge 1$ for $e' \gg 0$, and therefore 
$\alpha p^{e'} \le {\rm fpt}(f)p^{e'} -1$. 
Then, since we have  
\[
\alpha \le 
\displaystyle\frac{{\rm fpt}(f)p^{e'}-1}{p^{e'}}
<
\frac{{\rm fpt}(f)p^{e'}}{p^{e'}}
= 
{\rm fpt}(f)
\le 
\displaystyle\frac{\mu_f(p^{e'})}{p^{e'}},
\]
we obtain 
\[
\alpha \le \displaystyle\frac{\mu_f(p^{e'})-1}{p^{e'}} .
\]
Therefore,  
\[
\xi_f(\alpha) 
\ge 
\xi_f\left(\displaystyle\frac{\mu_f(p^{e'})-1}{p^{e'}}\right) \\
\underset{\rm{by}\;\eqref{4}}{\ge}
\lim_{e \to \infty}
  C_{e,\big\lceil \frac{\mu_f(p^{e'})-1}{p^{e'}}p^e\big\rceil}\\
\]
holds. 
We have 
$C_{e',\, \mu_f(p^{e'})-1} \neq 0$ 
by Lemma $\ref{mu}$.  
Since 
$\left\{ C_{e,\,\big\lceil \frac{\mu_f(p^{e'})-1}{p^{e'}}p^e \big\rceil } \right\}_{e\ge 0}$ 
is an increasing sequence, we obtain
$
\displaystyle\lim_{e \to \infty} 
C_{e,\,\big\lceil \frac{\mu_f(p^{e'})-1}{p^{e'}}p^e \big\rceil }
> 0
$. 
Therefore, $\xi_f(\alpha)>0$ 
for all 
$\alpha$ 
such that $\alpha < {\rm fpt}(f)$, as desired. 
\end{proof}

Next, we shall show $5)$. 

\begin{proof}
Let 
$F=
\left\{ 
\alpha \in \left[ \frac{a}{p^e}, \frac{a+1}{p^e} \right] 
\;\Big|\; \text{$\alpha$ is a discontinuity for $\xi_f(x)$}
\right\} $ 
and 
$\Omega = \left[ \frac{a}{p^e}, \frac{a+1}{p^e} \right] - F$. 
Recall that $F$ is a countable set,
and $\displaystyle\lim_{e\to \infty}\xi_{f,e}(\alpha)=\xi_{f}(\alpha)$ for any $\alpha\in\Omega$ by Thorem $\ref{th1}$ $1)$, $2)$.
Then, we have 
\begin{eqnarray*}
\displaystyle\int_{\frac{a}{p^e}}^{\frac{a+1}{p^e}} \xi_f(x)dx
&=& \displaystyle\int_{\Omega} \xi_f(x)dx \\
&=& \displaystyle\int_{\Omega} \lim_{e\to \infty}\xi_{f,e}(x)dx \\
&=& \displaystyle\lim_{e\to \infty}\int_{\Omega} \xi_{f,e}(x)dx \\
&=& \displaystyle\lim_{e\to \infty}\int_{\frac{a}{p^e}}^{\frac{a+1}{p^e}} \xi_{f,e}(x)dx \\
&=& \displaystyle\frac{1}{p^e} C_{e,\, a}
\end{eqnarray*}
by Lebegue's dominated convergence theorem, as desired. 
\end{proof}


We shall show $6)$.

\begin{proof}
Let $g,h:{\mathbb N}\rightarrow {\mathbb R}$ be functions.
If there exists a positive constant $C$ such that
$|h(n)| \le C g(n)$ for $n \gg 0$, 
then we write $h(n)=O(g(n))$.
If $R/(f)$ is normal, then there exists $\beta (R/(f)) \in \mathbb{R}$ such that
\[
e_{HK}(R/(f))p^{ne} + \beta(R/(f))p^{(n-1)e}= \ell_{R}(M_{e,\, 0}) + O(p^{(n-2)e})
\]
by Huneke-McDermott-Monsky \cite{HMM04}.
Since a hypersurface is Gorenstein,
that $\beta (R/(f))=0$ follows from Corollary $1.4$ in Kurano \cite{K06}.
Therefore, we have 
\begin{equation}\label{eHKmulti}
e_{HK}(R/(f))p^{ne} = \ell_{R}(M_{e,\, 0}) + O(p^{(n-2)e}). 
\end{equation}
First, we shall show that
\[
\left|
\frac{\xi_{f}\left(\frac{1}{p^{s}}\right) - \xi_{f}(0)}{\frac{1}{p^{s}}}
\right|
\longrightarrow
0
\;\;(s \rightarrow \infty).
\]
Since the sequence $\left\{C_{s+i,\,p^{i}}\right\}_{i \ge 0}$ is increasing,
we have 
\[
\xi_{f} \left( \frac{1}{p^{s}} \right)
= \limsup_{e \to \infty}C_{e,\,\lfloor p^{e-s} \rfloor} \ge C_{s,\,1}.
\]
Hence, we obtain
\[
\left|
\frac{\xi_{f}\left(\frac{1}{p^{s}}\right) - \xi_{f}(0)}{\frac{1}{p^{s}}}
\right|
=
\frac{\xi_{f}(0) - \xi_{f}\left(\frac{1}{p^{s}}\right)}{\frac{1}{p^{s}}}
\le
\frac{\xi_{f}(0) - C_{s,\,1}}{\frac{1}{p^{s}}}.
\]
Set $\lambda_{i}(e)$ as $e_{HK}(R/(f)) p^{en} - \ell_{R}(M_{e,\, i})$
for each $e \ge 0$ and $0 \le i \le p-1$.
Note that
\[
0\le\lambda_{0}(e)\le\lambda_{1}(e)\le\cdots\le\lambda_{p-1}(e).
\]
Since we have,
\[
p\times \frac{\ell_{R}(M_{s-1,\,0})}{p^{(s-1)n}}
=
\frac{\ell_{R}(M_{s,\, 0})}{p^{sn}}
+ \frac{\ell_{R}(M_{s,\, 1})}{p^{sn}}
+\cdots
+ \frac{\ell_{R}(M_{s,\, p-1})}{p^{sn}}
\]
for any $s \ge 1$ by $\eqref{2}$, then we obtain
\[
p\times \frac{\lambda_{0}(s-1)}{p^{(s-1)n}}
=
\frac{\lambda_{0}(s)}{p^{sn}}
+ \frac{\lambda_{1}(s)}{p^{sn}}
+\cdots
+ \frac{\lambda_{p-1}(s)}{p^{sn}}.
\]
Hence, since
\[
p\times \frac{\lambda_{0}(s-1)}{p^{(s-1)n}}
\ge
\frac{\lambda_{1}(s)}{p^{sn}},
\]
it holds that
\[
p^{2} \times  \frac{\lambda_{0}(s-1)}{p^{(s-1)(n-1)}}
\ge
\frac{\lambda_{1}(s)}{p^{s(n-1)}}
\ge
0.
\]
Therefore,
\begin{eqnarray*}
\frac{\xi_{f}(0) - C_{s,\,1}}{\frac{1}{p^{s}}}
&=&
\frac{p^{s}}{p^{sn}}\left(e_{HK}(R/(f))p^{sn}-C_{s,\,1}\times p^{sn}\right)\\
&=&
\frac{\lambda_{1}(s)}{p^{s(n-1)}} \\
&\le&
p^{2} \times \frac{\lambda_{0}(s-1)}{p^{(s-1)(n-1)}} \\
&=&
\frac{p^{2}}{p^{s-1}}\times \frac{\lambda_{0}(s-1)}{p^{(s-1)(n-2)}}\\
&\rightarrow&
0 \;\;\; (e \to \infty)
\end{eqnarray*}
by the equation $\eqref{eHKmulti}$.
Consequently, for any positive real number $\varepsilon >0$,
there exists a nutural number $s_{0} \in {\mathbb N}$
such that $s \ge s_{0}$ implies that 
\[
\left|
\frac{\xi_{f}\left(\frac{1}{p^{s}}\right) - \xi_{f}(0)}{\frac{1}{p^{s}}}
\right|
<
\frac{\varepsilon}{p}.
\]
Let $\delta = \displaystyle\frac{1}{p^{s_{0}}}$.
If $0<x<\delta$,
then there exists $s \in {\mathbb N}$ such that 
\[
\frac{1}{p^{s+1}} < x < \frac{1}{p^{s}} \le \frac{1}{p^{s_{0}}}.
\]
Therefore, 
\begin{eqnarray*}
\left|
\frac{\xi_{f}(x) - \xi_{f}(0)}{x}
\right|
&=&
\frac{\xi_{f}(0) - \xi_{f}(x)}{x} \\
&\le&
\frac{\xi_{f}(0) - \xi_{f}\left(\frac{1}{p^{s}}\right)}{\frac{1}{p^{s+1}}}\\
&\le&
p \times \frac{\varepsilon}{p}\\
&=&
\varepsilon
\end{eqnarray*}
as desired.
\end{proof}


Finally, we shall prove the last half of $3)$. 

\begin{Definition}\label{matfac}
\begin{rm}
Let $(S,{\mathfrak n})$ be a $(d+1)$-dimensional regular local ring. 
Let $0 \neq \alpha \in {\mathfrak n}$. 
The pair $(\rho,\sigma)$ is called a $matrix$ $factorization$ of the element $\alpha$ if 
all of the following conditions are satisfied:
\end{rm}
\begin{enumerate}[$(1)$]
\item
$\rho: G\rightarrow F$ and $\sigma:F\rightarrow G$ are $S$-homomorphisms, 
where $F$ and $G$ are finitely generated $A$-free modules, 
and ${\rm rank}_S F= {\rm rank}_S G$. 

\item
$\rho \circ\sigma = \alpha \cdot id_F$. 

\item
$\sigma\circ \rho = \alpha \cdot id_G$. 
\end{enumerate}
Actually, if either $(2)$ or $(3)$ is satisfied, the other is satisfied. 
\end{Definition}

\begin{Definition}
\begin{rm}
Let $(S,{\mathfrak n})$ be a $(d+1)$-dimensional regular local ring, 
and let $0 \neq \alpha \in {\mathfrak n}$. 
Let $(\rho,\sigma)$ and $(\rho',\sigma')$ be matrix factorizations of $\alpha$. 
We regard 
$\rho$ and $\sigma$
as 
$r \times r$ 
matrixes with entries in $S$,
and $\rho'$ and $\sigma'$
as 
$r' \times r'$ 
matrixes with entries in $S$.
Then, we write 
\[
   (\rho,\sigma) \oplus (\rho',\sigma')
   =
   \left(
   \left( 
          \begin{array}{cc}
          \rho & 0 \\
          0 & \rho' 
          \end{array}
   \right), 
   \left( 
          \begin{array}{cc}
          \sigma & 0 \\
          0 & \sigma' 
          \end{array}
   \right)
   \right)
\]
which is a matrix factorization of $\alpha$. 
\end{rm}
\end{Definition}

\begin{Definition}
\begin{rm}
Let $(S,{\mathfrak n})$ be a $(d+1)$-dimensional regular local ring, 
and let $0 \neq \alpha \in {\mathfrak n}$. 
A matrix factorization
$(\rho,\sigma)$ of $\alpha$ is called $reduced$ 
if all the entries of $\rho$ and $\sigma$ are in $\mathfrak n$. 
\end{rm}
\end{Definition}

\begin{Remark}\label{MFrem}
\begin{rm}
Let $(S,{\mathfrak n})$ be a $(d+1)$-dimensional regular local ring, 
and let $0 \neq \alpha \in {\mathfrak n}$. 
Let the map 
$\alpha:S \rightarrow S$ 
be multiplication by 
$\alpha\in {\mathfrak n}$ on $S$. 
If $(\rho,\sigma)$ is a matrix factorization of $\alpha \in {\mathfrak n}$, 
then we can write
\[
   (\rho,\sigma)
   \simeq
   (\alpha, id_S)^{\oplus v} 
   \oplus (id_S, \alpha)^{\oplus u} 
   \oplus (\gamma_1, \gamma_2),
\]
where $v$ and $u$ are some integers, 
and $(\gamma_1, \gamma_2)$ is reduced. 
Therefore, 
\begin{eqnarray*}
  {\rm cok}(\rho) 
  &\simeq& {\rm cok}(\alpha)^{\oplus v} 
      \oplus {\rm cok}(id_S)^{\oplus u} \oplus {\rm cok}(\gamma_1) \\
  &\simeq& \left(S/(\alpha) \right)^{\oplus v} \oplus {\rm cok}(\gamma_1).
\end{eqnarray*}
It is known that 
${\rm cok}(\gamma_1)$ 
has no free direct summands 
if $(\gamma_1, \gamma_2)$ is reduced (\cite{E80}, Corollary $6.3$). 
Consequently, 
$v$ is equal to the largest rank of a free $S/(\alpha)$-module appearing as a direct summand of ${\rm cok}(\rho)$. 
\end{rm}
\end{Remark}

Let $F^e:R \rightarrow F^{e}_{*}R$ be the $e$-th Frobenius map. 
Consider the map $f: F^{e}_{*}R \rightarrow F^{e}_{*}R$.
We have $f=F^{e}_{*}(f^{p^{e}})=F^{e}_{*}(f)\cdot F^{e}_{*}(f^{p^{e}-1})=F^{e}_{*}(f^{p^{e}-1})\cdot F^{e}_{*}(f)$.
Therefore, $(F^{e}_{*}(f), F^{e}_{*}(f^{p^{e}-1}))$ is a matrix factorization.
We put
\[
(F^{e}_{*}(f), F^{e}_{*}(f^{p^{e}-1}))
=
(f, id_R)^{\oplus v_e} \oplus (id_R , f)^{\oplus u_e} \oplus (reduced).
\]
By Remark $\ref{MFrem}$ this implies that 
$v_e$ is the number of $R/(f)$ appearing as the direct summand of $\frac{F^{e}_{*}R}{F^{e}_{*}(f)(F^{e}_{*}R)}=F^{e}_{*}(R/(f))$. 
That is, 
${\displaystyle\lim_{e \to \infty}}\frac{v_e}{p^{en}}$ 
is the F-signature of $R/(f)$, 
denoted by $s(R/(f))$. 

\begin{Proposition}
$v_e = \ell_{R}(M_{e, \, p^e-1})$. 
\end{Proposition}

\proof
We can regard the map 
$F^{e}_{*}(f^{p^{e}-1}):F^{e}_{*}R \longrightarrow F^{e}_{*}R$ 
as a $p^{(n+1)e} \times p^{(n+1)e}$ matrix $A$ with entries in $R$;
\[
A = \left(
    \begin{array}{@{\,}ccc|ccc|ccc@{\,}}
     &        &   &   &        &    &   &        &    \\ 
      \;\hugesymbol{I_{v_e}}\;&   &   &   &        &    &   &        &    \\
      &        &  &   &        &    &   &        &    \\ \hline
      &        &   & f &        &    &   &        &    \\  
      &        &   &   & \ddots &    &   &        &    \\
      &        &   &   &        & f  &   &        &    \\ \hline
      &        &   &   &        &    &   &       &    \\
      &        &   &   &        &    & \,\;\hugesymbol{B}&       &    \\
      &        &   &   &        &    &   &        &    
    \end{array}
    \right)
    ,                   
\]
where $I_{v_e}$ is the identity matrix of size $v_{e}$, 
and $B$ is a matrix with entries in $\mathfrak m$. 
Therefore, 
we have 
\[
v_e
=
{\rm dim}_{R/{\mathfrak m}} \Big({\rm Im}\big(R/{\mathfrak m} \otimes F^{e}_{*}(f^{p^{e}-1}) \big) \Big)
=
{\rm dim}_{R/{\mathfrak m}} 
F^{e}_{*}
\left( 
       \frac{(f^{p^e-1})+ {\mathfrak m}^{[p^{e}]}}{{\mathfrak m}^{[p^{e}]}}
\right)
=
M_{e, \, p^{e}-1}.
\]
\qed

We completed a proof of Theorem $\ref{th1}$.

\begin{Remark}
\begin{rm}
Let $(S, \mathfrak{n}, k)$ be a complete regular local ring of characteristic $p>0$. Suppose that $k$ is perfect. 
Let $I$ be a ideal of $S$, and put $\overline{S}=S/I$.
Suppose that $a_e$ is equal to the largest rank of a free $\overline{S}$-module appearing in a direct summand of $F^e_*\overline{S}$.
Then it is known that
\[
a_e = 
\dim_k \frac{(I^{[p^e]} : I) + {\mathfrak m}^{[p^e]}}
{{\mathfrak m}^{[p^e]}}
\]
by Fedder's lemma (see \cite{AE05}).
If $I=(f)$, then
\[
a_e = 
\dim_k \frac{(f^{p^{e}-1}) + {\mathfrak m}^{[p^e]}}
{{\mathfrak m}^{[p^e]}}.
\]
\end{rm}
\end{Remark}

\section{Examples}\label{Example} 

 Let $f=X^{\alpha_1}_1 X^{\alpha_2}_2 \cdots X^{\alpha_{n+1}}_{n+1}$ and 
 $\alpha_1 \le \alpha_2 \le \dots \le \alpha_{n+1}$. 
 We set 
 $(\underline{X}^{p^e})=(X^{p^e}_1,\, X^{p^e}_2, \dots ,X^{p^e}_{n+1})$ 
 for $e\ge0$. 
 We have an exact sequence
 \[
    0 \longrightarrow 
    M_{e,\, t}
    \longrightarrow \displaystyle\frac{R}
    {(f^{t+1})+(\underline{X}^{p^e})}
    \longrightarrow \frac{R}
    {(f^t)+(\underline{X}^{p^e})}
    \longrightarrow 0
 \]
 for any $t \ge 0$.
 On the other hand, we have an exact sequence
 \[
    0 \longrightarrow 
    \displaystyle\frac{(f^t)+(\underline{X}^{p^e})}
     {(\underline{X}^{p^e})} 
    \longrightarrow \frac{R}{(\underline{X}^{p^e})}
    \longrightarrow \frac{R}
    {(f^t)+(\underline{X}^{p^e})}
    \longrightarrow 0
 \]
 for any $t \ge 0$. 
 Hence, 
 \[
    \ell_R\left( \displaystyle\frac{R}{(f^t)+(\underline{X}^{p^e})} \right)
    =
    \begin{cases}
    p^{e(n+1)} - \displaystyle\prod_{j=1}^{n+1} (p^e - t\alpha_j) & 
    (\text{if $t\alpha_{n+1} < p^e$}) \\
    
    p^{e(n+1)} & (\text{otherwise})
    \end{cases}
 \]
 holds. 
 Therefore, we have 
 \begin{equation}\label{ex}
    \ell_R(M_{e,\, t})
    =
    \begin{cases} \medskip
    0 & \left(\displaystyle\frac{p^e}{\alpha_{n+1}} \le t \right) \\
    \smallskip
    \displaystyle\prod_{j=1}^{n+1} (p^e - t\alpha_j) 
    & \left(\displaystyle\frac{p^e}{\alpha_{n+1}}-1 
      \le t < \displaystyle\frac{p^e}{\alpha_{n+1}} \right) \\
    
    \displaystyle\prod_{j=1}^{n+1} (p^e - t\alpha_j)-
                              \prod_{j=1}^{n+1} (p^e - (t+1)\alpha_j)
     &  \left(t< \displaystyle\frac{p^e}{\alpha_{n+1}}-1 \right) \\
    \end{cases}.
 \end{equation}
 If $t< \displaystyle\frac{p^e}{\alpha_{n+1}}-1$, 
 \begin{eqnarray*}
     \ell_R(M_{e,\, t}) &=& \prod_{j=1}^{n+1} (p^e - t\alpha_j)-
                              \prod_{j=1}^{n+1} (p^e - (t+1)\alpha_j)\\
                      &=& \sum_{j=1}^{n+1} (-1)^{j} t^j \beta_j p^{e(n+1-j)} - 
                          \sum_{j=1}^{n+1} (-1)^{j} (t+1)^j \beta_j p^{e(n+1-j)}\\
                      &=& \sum_{j=1}^{n+1} (-1)^{j+1} 
                            \left(\sum_{i=0}^{j-1} \dbinom{j}{i}t^i \right)
                            \beta_{j} p^{e(n+1-j)}, 
 \end{eqnarray*}
 where $\beta_j$ 
 denotes the elementary symmetric polynomial of degree $j$ 
 in $\alpha_1,\, \alpha_2, \dots ,\, \alpha_{n+1}$.
 Hence
 \[
       C_{e,\, t} = \displaystyle\frac{\ell_{R}(M_{e,\, t})}{p^{en}} 
                  = \sum_{j=1}^{n+1} (-1)^{j+1} 
                            \left(\sum_{i=0}^{j-1} 
                              \dbinom{j}{i}\frac{t^i}{p^{e(j-1)}} \right)
                            \beta_{j} 
 \]
 holds. 
 We shall calculate $\xi_f(x)$. 
 If $x< \displaystyle\frac{1}{\alpha_{n+1}}$, 
 then $\lfloor xp^e \rfloor < \displaystyle\frac{p^e}{\alpha_{n+1}} -1$ 
 for $e \gg 0$.
 Then,
  \[
    \displaystyle C_{e,\, \lfloor xp^e \rfloor }
    =
    \sum_{j=1}^{n+1} (-1)^{j+1} 
                            \left(\sum_{i=0}^{j-1}
                            \displaystyle
                              \dbinom{j}{i}
                              \frac{\lfloor xp^e \rfloor^i}
                                    {p^{e(j-1)}} \right)
                            \beta_{j} .
 \]
 Since 
 $xp^e-1 \le \lfloor xp^e \rfloor \le xp^e$,
 we have
 \[
      \displaystyle\lim_{e\to\infty} \frac{\lfloor xp^e \rfloor ^a}{p^{eb}} 
    =
    \begin{cases}
         x^{a}  & (\text{if $a=b$})\\
         0  &  (\text{if $a<b$})
    \end{cases}.
 \]
 Consequently, 
 \begin{eqnarray}\label{9}
    \xi_f(x)
     =
     \beta_1 -2\beta_2 x +3\beta_3 x^2 
     - \cdots +(-1)^{n}(n+1)\beta_{n+1} x^n
 \end{eqnarray}
 holds for $0\le x< \displaystyle\frac{1}{\alpha_{n+1}}$. 
 In particular, 
 $e_{HK}(R/(f)) = \xi_{f}(0) = \alpha_1 + \alpha_2 + \cdots + \alpha_{n+1}$. 
 By $\eqref{ex}$, we have
 \begin{eqnarray}\label{10}
 \xi_{f}(x)=0
 \end{eqnarray}
 if $x> \displaystyle\frac{1}{\alpha_{{n+1}}}$.

 Next, we shall calulate 
 $\xi_f\left(\displaystyle\frac{1}{\alpha_{n+1}}\right)$.
 Since
 $\displaystyle\frac{p^e}{\alpha_{n+1}}-1
 \le
 \left\lfloor\displaystyle\frac{p^e}{\alpha_{n+1}}\right\rfloor
 \le\displaystyle\frac{p^e}{\alpha_{n+1}}$ 
 for any $e \ge 0$,

 \begin{eqnarray*}
 \ell (M_{e, \, \big\lfloor  \frac{1}{\alpha_{n+1}}p^e \big\rfloor}) 
 &=& \displaystyle\prod_{j=1}^{n+1} 
      \left\{ 
             p^e - 
             \bigg\lfloor \frac{p^e}{\alpha_{n+1}} \bigg\rfloor 
             \alpha_{j}
      \right\} \\
  &=&  \varepsilon_e \displaystyle\prod_{j=1}^{n} 
      \left\{
             p^e - 
             (p^e- \varepsilon_e)\frac{\alpha_{j}}{\alpha_{n+1}}
      \right\} \\
  &=&  \varepsilon_e \displaystyle\prod_{j=1}^{n} 
      \left\{
             \left(1-\frac{\alpha_{j}}{\alpha_{n+1}} \right)p^e
              +  \frac{\varepsilon_e}{\alpha_{n+1}}
              \alpha_{j}
      \right\} \\
  &=&  \varepsilon_e 
       \displaystyle\left(\frac{1}{\alpha_{n+1}}\right)^n
       \prod_{j=1}^{n} 
       \big\{
             (\alpha_{n+1}-\alpha_{j})p^e
              + \varepsilon_e
              \alpha_{j}
      \big\} \\
  &=&  \varepsilon_e 
       \displaystyle\left(\frac{1}{\alpha_{n+1}}\right)^n
       \left\{  
               p^{en}\displaystyle\prod_{j=1}^{n}(\alpha_{n+1}-\alpha_{j})
               +
               \sum_{k=1}^{n}
               \sum_{1\le i_1< i_2 < \cdots < i_k \le n}
               \delta_{\underline{i}}\,
               p^{e(n-k)}\varepsilon^k_e 
               \alpha_{i_1}\alpha_{i_2}\cdots\alpha_{i_k}
        \right\} \\
  &=&  \varepsilon_e 
       \displaystyle\left(\frac{1}{\alpha_{n+1}}\right)^n
       p^{en}
       \left\{  
               \displaystyle\prod_{j=1}^n(\alpha_{n+1}-\alpha_j)
               +
               \sum_{k=1}^n  \delta_k  
               \left(\frac{\varepsilon_e}{p^e} \right)^{k}
        \right\}, 
 \end{eqnarray*}
 where $\varepsilon_e \equiv  p^e \pmod{\alpha_{n+1}}$ 
 such that $0 \le \varepsilon_e <\alpha_{n+1}$, 
 and 
\begin{eqnarray*}
\delta_{\underline{i}}
&=&
   \prod_{j \neq i_1 , i_2 , \dots , i_k} (\alpha_{n+1}-\alpha_{j}),\\
\delta_{k}
&=&
    \sum_{1\le i_1< i_2 < \cdots < i_k \le n}
               \delta_{\underline{i}}\,
               \alpha_{i_1}\alpha_{i_2}\cdots\alpha_{i_k}.
  \end{eqnarray*}
Hence, 
\[
     C_{e, \,  \big\lfloor  \frac{1}{\alpha_{n+1}} p^e \big\rfloor}
     =
     \varepsilon_e 
       \displaystyle\left(\frac{1}{\alpha_{n+1}}\right)^n
       \left\{  
               \displaystyle\prod_{j=1}^n(\alpha_{n+1}-\alpha_j)
               +
               \sum_{k=1}^n
               \delta_{k}
               \left(\frac{\varepsilon_e}{p^e} \right)^{k}
        \right\} ,
\]
and therefore
\begin{equation}\label{supC}
    \displaystyle
     \limsup_{e\to\infty}
     C_{e, \,  \big\lfloor  \frac{1}{\alpha_{n+1}} p^e \big\rfloor}
     =
     \left(\limsup_{e\to\infty}\varepsilon_e \right)
       \displaystyle\left(\frac{1}{\alpha_{n+1}}\right)^n
               \displaystyle\prod_{j=1}^n(\alpha_{n+1}-\alpha_j).
\end{equation}

We shall examine whether
$\displaystyle\lim_{e\to\infty}\varepsilon_{e}$
exists.
Let $\alpha_{n+1} = p^{s}q$,
where $q$ is coprime to $p$, and $s$ is a non-negative integer.
If $p \equiv 1 \pmod{q}$,
then we can find that
$\varepsilon_{e}$ is constant for any $e \ge s$
by the Chinese remainder theorem.
If $p \not\equiv 1 \pmod{q}$,
then $\varepsilon_{e}$ is eventually periodic with period more than $1$.

 From the following Proposition $\ref{form}$, 
 we get to know the function $\xi_f(x)$.

 \begin{Proposition}\label{form}
    \begin{rm}
    Let $f=X^{\alpha_1}_1 X^{\alpha_2}_2 \cdots X^{\alpha_{n+1}}_{n+1}$ with 
 $\alpha_1 \le \alpha_2 \le \dots \le \alpha_{n+1}$.  
    \begin{enumerate}[1)]
    
      \item
      We have ${\rm fpt}(f)= \displaystyle\frac{1}{\alpha_{n+1}}$. 
      If $\alpha_{n+1} \ge 2$, we have
      $s\big(R/(f) \big) =0$.
      
      \item
      $\displaystyle\lim_{ x \to \frac{1}{\alpha_{n+1}} -0 }\xi_f(x)
       = \displaystyle\left(\frac{1}{\alpha_{n+1}}\right)^{n-1}
           \prod_{j=1}^n (\alpha_{n+1} - \alpha_j) \ge 0$. 
      
      \item
      The function $\xi_f(x)$ is continuous on $[0,1]$ if and only if 
      $\alpha_{n+1}=\alpha_{n}$ holds.
      
      \item
      Let $\alpha_{n+1} = p^{s}q$,
      where $q$ is coprime to $p$, and $s$ is a non-negative integer.
      The limit
      $\displaystyle\lim_{e\to\infty}\xi_{f,e}\left(\frac{1}{\alpha_{n+1}}\right)$
      exists if and only if
      it satisfies that $\alpha_{n+1}=\alpha_{n}$ or $p \equiv 1 \pmod{q}$.
    \end{enumerate}
    \end{rm}
 \end{Proposition}

\proof
\begin{rm}
By $\eqref{9}$ and $\eqref{10}$, we obtain $1)$ immeiately.

Next we shall prove $2)$.
We set 
\[
g(x)=\beta_1 -2\beta_2 x +3\beta_3 x^2 
      - \cdots +(-1)^{n}(n+1)\beta_{n+1} x^n
\]
and
\[
h(x) = (x - \alpha_{1})(x - \alpha_{2})\cdots (x-\alpha_{n+1}).
\]
Now, since 
$h(x) = x^{n+1} - \beta_{1}x^{n} +\beta_{2}x^{n-1} - \cdots +(-1)^{n+1}\beta_{n+1}$, 
\[
x^{n+1} h\left( \displaystyle\frac{1}{x} \right) 
  =  1- \beta_{1}x + \beta_{2}x^2 - \cdots +(-1)^{n+1}\beta_{n+1} x^{n+1}.
\]
Hence, we have the following equation
\[
g(x) 
=
-\left\{ x^{n+1}h\left( \frac{1}{x} \right) \right\}'
=
 -(n+1) x^n h\left( \displaystyle\frac{1}{x} \right) 
          + x^{n-1} h' \left( \displaystyle\frac{1}{x} \right) .
\]
Since $h(\alpha_{n+1})=0$, 
\begin{eqnarray*}
\displaystyle\lim_{ x \to \frac{1}{\alpha_{n+1}} -0 }\xi_f(x)
 &=& g\left(\frac{1}{\alpha_{n+1}}\right) \\
 &=& \left(\frac{1}{\alpha_{n+1}}\right)^{n-1} h'(\alpha_{n+1}) \\
 &=& \left( \frac{1}{\alpha_{n+1}} \right)^{n-1} 
                \prod_{j=1}^{n} (\alpha_{n+1} - \alpha_{j}) \ge 0 .
\end{eqnarray*}
The assertion $3)$ follows from $\eqref{ex}$, $\eqref{9}$ and $2)$ as above.
The assertion $4)$ follows from the equation $\eqref{supC}$.
\end{rm}
\qed

\begin{Example}
\begin{rm}
If 
$\alpha_1 = \alpha_2 = \cdots =\alpha_{n-2}=0$ 
and $\alpha_{n-1}\neq 0$, 
the derivative 
\[
g'(x)=
-2(\alpha_{n+1}\alpha_{n}+\alpha_{n+1}\alpha_{n-1}+\alpha_{n}\alpha_{n-1})
+6\alpha_{n+1}\alpha_{n}\alpha_{n-1}x.
\]
Let $\alpha$ be the root of $g'(x)=0$, that is,
\[
\alpha = \displaystyle\frac{1}{3}
\times\frac{\alpha_{n+1}\alpha_{n}+\alpha_{n+1}\alpha_{n-1}+\alpha_{n}\alpha_{n-1}}{\alpha_{n+1}\alpha_{n}\alpha_{n-1}}. 
\]
Then, we have
\[
\alpha- \displaystyle\frac{1}{\alpha_{n+1}}
=
\frac{1}{\alpha_{n+1}}
\left\{
       \frac{1}{3}\left(
       \frac{\alpha_{n+1}}{\alpha_{n-1}}
       +\frac{\alpha_{n+1}}{\alpha_{n}}
       +1
       \right)
       -1
\right\}
\ge
0,
\]
and so $g'(x)< 0$ for any $x<\displaystyle\frac{1}{\alpha_{n+1}}$. 
Moreover, 
if $\alpha_{n+1}\neq \alpha_{n}$ 
we obtain 
$g'\left(\displaystyle\frac{1}{\alpha_{n+1}}\right) < 0$. 
The second derivative 
$g''(x)$ 
is positive for any $x\in {\mathbb R}$. 
In fact, 
$g''(x)=6\alpha_{n-1}\alpha_{n}\alpha_{n+1} > 0$. 
\end{rm}
\end{Example}

\section*{Acknowledgements}
This paper was written while the author was a student in a doctor's course of Meiji University.
I would like to express my deepest gratitude to my advisor Prof. Kurano who provided helpful comments and suggestions.
I would also like to thank Yuji Kamoi, Akiyoshi Sannai and Kei-ichi Watanabe who gave me invaluable comments and warm encouragements.

\end{document}